\documentclass[a4paper,11pt]{amsart}
\usepackage[english]{babel}
\usepackage[utf8]{inputenc}
\usepackage{paralist,url,verbatim}
\usepackage{amscd}
\usepackage{bm}
\usepackage{amsmath,amsthm,amssymb,amsfonts}
\usepackage{color,graphicx}
\usepackage{graphicx, graphics,subfigure}
\usepackage[bookmarksopen=true]{hyperref}
\usepackage{tablefootnote}

\theoremstyle{plain}
\newtheorem{theorem}{\bf Theorem}[section]

\newtheorem{lemma}[theorem]{Lemma}
\newtheorem{proposition}[theorem]{Proposition}

\theoremstyle{definition}
\newtheorem{remark}[theorem]{Remark}
\newtheorem{definition}[theorem]{Definition}

\newtheorem{algorithm}[theorem]{Algorithm}

\newcommand{\kk}{\Bbbk }
\newcommand{\depth}{\operatorname{depth} }
\newcommand{\reg}{\operatorname{reg} }

\newcommand{\pdim}{\operatorname{projdim} }
\newcommand{\Ker}{\operatorname{Ker} }

\newcommand{\lcm}{\operatorname{lcm} }
\newcommand{\sign}{\operatorname{sign} }

\renewcommand{\Im}{\operatorname{Im}}

\newcommand{\C}{\mathcal{C}}
\newcommand{\LL}{\mathcal{L}}
\newcommand{\K}{\mathcal{K}}

\newcommand{\ZZ}{\mathbb{Z}}
\newcommand{\BB}{\mathbb{B}}

\newcommand{\NN}{\mathbb{N}}

\newcommand{\Lyu}{\operatorname{Lyu} }

\newcommand{\Hom}{\mathrm{Hom}}

\newcommand{\D}{\Delta}

\begin{document}

\title{An alternative algorithm for computing the Betti table of a monomial ideal}
\author[M.~Torrente]{Maria-Laura Torrente}
\address{ Dipartimento di Matematica,
 Universit\`a di Genova,
 Via Dodecaneso 35, Genova 16146, Italy}
\email{torrente@dima.unige.it}

\author[M.~Varbaro]{Matteo Varbaro}
\address{ Dipartimento di Matematica,
 Universit\`a di Genova,
 Via Dodecaneso 35, Genova 16146, Italy}
\email{varbaro@dima.unige.it}

\date{}
\maketitle

\begin{abstract}
In this paper we develop a new technique to compute the Betti table of a monomial ideal.
We present a prototype implementation of the resulting algorithm and we perform numerical experiments suggesting a very promising efficiency. On the way of describing the method, we also prove new constraints on the shape of the possible Betti tables of a monomial ideal.
\end{abstract}

\section{Introduction}

Since many years syzygies, and more generally free resolutions, are central in purely theoretical aspects of algebraic geometry;
more recently, after the connection between algebra and statistics have been initiated by Diaconis and Sturmfels in \cite{DS}, free resolutions have also become an important tool in statistics 
(for instance, see \cite{D, SW}).  
As a consequence, it is fundamental to have efficient algorithms to compute them. 
The usual approach uses Gr\"obner bases and exploits a result of Schreyer (for more details see 
\cite{Sc1,Sc2} or \cite[Chapter 15, Section 5]{Ei}). The packages for free resolutions of the most used computer algebra 
systems, like \cite{M2, singular, cocoa}, are based on these techniques. 
In this paper, we introduce a new algorithmic method to compute the minimal graded free resolution of any finitely generated graded module over a polynomial ring such that some (possibly nonminimal) graded free resolution is known a priori. We describe this method and we present the resulting algorithm in the case of monomial ideals in a polynomial ring, in which situation we always have a starting nonminimal graded free resolution. 
A first implementation of the algorithm has already been tested and has led to  promising results. It is our opinion that, 
by accurately refining the implementation, it is possible to achieve substantially better performance/timings than the already good ones shown 
in Table \ref{table}.

Monomial ideals are, essentially, combinatorial objects, however the combinatorics of their minimal free resolutions is still mysterious. In the literature, the minimal free resolution of several classes of monomial ideals have been studied and fully understood. 
Many of these resolutions are {\it cellular}, i.e. they come from a (labeled) regular cell complex built by the generators of the monomial ideal (see \cite{BS} or \cite[Chapter 4]{MS}). Due to this reason, cellular resolutions have been intensively studied in the recent literature, giving rise to a new interesting area of research. However, not every monomial ideal admits a minimal cellular resolution, though all monomial ideals admit some (not necessarily minimal) cellular resolution. In this paper, we  use algebraic discrete Morse theory to algorithmically reduce a cellular resolution of a given monomial ideal to its minimal free resolution. We explicitly describe the algorithm when the starting cellular resolution is the free resolution found by Lyubeznik in \cite{Ly}, which is actually simplicial (i.e. the regular cell complex is simplicial). Our implementation of the algorithm is provided in this case.

On the way of writing the algorithm, we prove some rigidity properties, like the following vanishing for monomial ideals generated in degrees $\leq d$ (Theorem \ref{thm:subadditivity}):
\[\beta_{i,k}(S/I)=0 \ \ \forall \ k=j,\ldots ,j+d-1 \ \ \ \implies \ \ \ \beta_{i+1,j+d}(S/I)=0. \] 
This strengthens the recent result proved by Herzog and Srinivasan in \cite{HS} saying that, if $\beta_{i,k}(S/I)=0  \ \forall \ k\geq j$, then $\beta_{i+1,k+d}(S/I)=0   \ \forall \ k\geq j$. Theorem \ref{thm:subadditivity} has independently been proved, by different techniques, in \cite{Ya}. 

Finally, a similar idea could be used to compute the simplicial homology with coefficients in a field of a finite simplicial complex (see Remark \ref{rem:simplicial homology} for more details). Already many algorithms are available for this task; nevertheless, we think it would be worth to inquire on the potentiality of this idea.

\section{Reducing resolutions}\label{RedRes}

Let $R$ be a ring, and ${\bf F}_{\bullet}=(F_i,\partial_i)_{i\in\NN}$ a complex of free $R$-modules, where $\NN$ is the set of natural numbers $\{0,1,2,\ldots \}$. For any $i\in\NN$ let $\Delta_i$ be a set to index the basis of $F_i$, that is:
\[F_i=\bigoplus_{\sigma \in \Delta_i}R\cdot \sigma.\]
With respect to such a basis, we can write the differentials $\partial_j:F_j\rightarrow F_{j-1}$ as:
\[1_\sigma\mapsto \sum_{\tau\in \Delta_{j-1}}[\sigma:\tau]\cdot 1_\tau,\]
where $[\sigma:\tau]$ is an element of the ring $R$.

We are going to describe a special case of the so-called Algebraic Discrete Morse Theory, an useful tool to reduce the size of ${\bf F}_{\bullet}$ (when possible). Let $\alpha\in \Delta_j$ and $\beta\in \Delta_{j-1}$ such that $[\alpha:\beta]$ is an {\it invertible element in $R$}. Then construct the following complex ${\bf \tilde{F}}_{\bullet}=(\tilde{F}_i,\tilde{\partial}_i)_{i\in\NN}$ of free $R$-modules:
\begin{align*}
\tilde{F}_i =
\begin{cases}
\bigoplus_{\sigma \in \Delta_j\setminus\{\alpha\}}R\cdot \sigma & \text { if } i=j,\\
\bigoplus_{\tau \in \Delta_{j-1}\setminus\{\beta\}}R\cdot \tau & \text { if } i=j-1,\\
F_i & \text {otherwise}
\end{cases}
\end{align*} 
The differentials $\tilde{\partial}_i:\tilde{F}_i\rightarrow \tilde{F}_{i-1}$ are defined as:
\begin{align*}
\tilde{\partial}_i(1_\sigma) =
\begin{cases}
\partial_i(1_\sigma) & \text { if } i\notin \{j+1,j\},\\
\sum_{\tau\in \Delta_j\setminus\{\alpha\}}[\sigma:\tau]\cdot 1_\tau& \text { if } i=j+1,\\
\sum_{\tau\in \Delta_{j-1}\setminus\{\beta\}}\left([\sigma:\tau]-\frac{[\sigma:\beta][\alpha:\tau]}{[\alpha:\beta]}\right)\cdot 1_\tau& \text { if } i=j
\end{cases}
\end{align*} 
The following lemma is a particular case of \cite[Theorem 2.2]{JW}.
\begin{lemma}\label{thelemma}
${\bf \tilde{F}}_{\bullet}=(\tilde{F}_i,\tilde{\partial}_i)_{i\in\NN}$ is a complex of free $R$-modules homotopically equivalent to  ${\bf F}_{\bullet}=(F_i,\partial_i)_{i\in\NN}$. In particular, 
\[H_i({\bf F}_{\bullet})\cong H_i({\bf \tilde{F}}_{\bullet}) \ \ \ \forall \ i\in\NN.\]
\end{lemma}
Let us remind that ${\bf F}_{\bullet}$ is a free resolution of an $R$-module $M$ if $H_i({\bf F}_{\bullet})=0$ for any $i>0$ and $H_0({\bf F}_{\bullet})\cong M$. Therefore the above lemma implies that, if ${\bf F}_{\bullet}$ is a free resolution of $M$, then ${\bf \tilde{F}}_{\bullet}$ is a free resolution of $M$ as well. 

We are interested in the case that $R$ is a polynomial ring over a field $\kk$ and ${\bf F}_{\bullet}$ is a graded free resolution of a finitely generated graded $R$-module $M$. In this case, ${\bf F}_{\bullet}$ is not minimal if and only if there exist $\alpha\in \Delta_j$ and $\beta\in \Delta_{j-1}$ such that $[\alpha:\beta]\in \kk^{\times}$. Therefore Lemma \ref{thelemma}  provides an iterative procedure to get the minimal graded free resolution of $M$ if some graded free resolution of $M$ is known a priori.

\section{Simplicial resolutions}\label{SimplRes}

Let $S=\kk[x_1,\ldots ,x_n]$ be the polynomial ring in $n$ variables over a field $\kk$, and $I\subseteq S$ be a monomial ideal generated by monomials $\underline{u}=u_1,\ldots ,u_r$. Let $\Delta$ be a (labeled) simplicial complex on the set $\{1,\ldots ,r\}=:[r]$, where the label of a face $\sigma\in\D$ is the monomial $m_{\sigma}\in S$, where:
\begin{equation}\label{eq:msigma}
m_\sigma:=\lcm(u_i|i\in \sigma)\in S.
\end{equation}
Let us remind that the reduced simplicial homology of $\Delta$ (with coefficients in $\kk$) is the homology of the complex of $\kk$-vector spaces
${\bf C_{\bullet}}(\Delta ;\kk)=(C_i,\partial_i)_{i=-1,\ldots ,d-1}$, where $d-1$ is the dimension of $\D$, 
\[C_i=\bigoplus_{\sigma}\kk\cdot \sigma\]
where the sum runs over the $i$-dimensional faces of $\D$, and the differentials are defined as:
\[1_\sigma\mapsto \sum_{v\in\sigma}\sign(v,\sigma)\cdot 1_{\sigma\setminus \{v\}},\]
where 
\begin{equation}\label{eq:sign}
\sign(v,\sigma):=(-1)^{q-1}\in\kk
\end{equation}
 if $v$ is the $q$th element of $\sigma$. By imitating this construction, let us consider the graded complex of free $S$-modules ${\bf F_{\bullet}}(\D;\underline{u})=(F_i,\partial_i)_{i=0,\ldots ,d}$, where 
\[F_i=\bigoplus_{\sigma}S(-\deg(m_{\sigma}))\]
where the sum runs over the $(i-1)$-dimensional faces of $\D$, and the differentials are defined as:
\[1_\sigma\mapsto \sum_{v\in\sigma}\sign(v,\sigma)\frac{m_{\sigma}}{m_{\sigma\setminus\{v\}}}\cdot 1_{\sigma\setminus \{v\}}.\]
Since any $v\in [r]$ is a vertex of $\D$, one has that $H_0({\bf F_{\bullet}}(\D;\underline{u}))=S/I$. Furthermore, because the differential are graded of degree 0, one can check that ${\bf F_{\bullet}}(\D;\underline{u})$ is a graded free resolution of $S/I$ if and only if $\D_{\leq u}$ is an acyclic simplicial complex for any monomial $u\in S$, where $\D_{\leq u}$ consists in the faces $\sigma\in\D$ such that $m_{\sigma}|u$ (cf. \cite[Lemma 2.2]{BPS}). We will call this kind of resolutions, introduced for the first time by Bayer, Peeva and Sturmfels in \cite{BPS}, {\it simplicial}.
In particular, if $\D=\K$ is the $(r-1)$-simplex, then ${\bf F_{\bullet}}(\K;\underline{u})$ is a graded free resolution of $S/I$: indeed in this case $\D_{\leq u}$ is the simplex on the vertices $i$ such that $u_i|u$, which is contractible. That ${\bf F_{\bullet}}(\K;\underline{u})$ was a resolution was proved, for the first time, by Taylor in \cite{Ta}.

\begin{definition}
With the above notation, ${\bf F_{\bullet}}(\K;\underline{u})$ is called the {\it Taylor resolution}.
\end{definition}

The Taylor resolution is far to be minimal in general: in fact, it tends to be huge. 
A smaller graded free resolution, still simplicial, was found by Lyubeznik in \cite{Ly}. For each $k\in [r]$ and $\sigma\subseteq [r]$, denote by 
\begin{equation}\label{eq:sigma>k}
\sigma_{>k}:=\{i\in \sigma:i>k\}.
\end{equation}
Let $\LL$ be the simplicial complex on $[r]$ consisting of those faces $\sigma\subseteq [r]$ such that $u_k$ does not divide $m_{\sigma_{>k}}$ for any $k\in[r]$. Then Lyubeznik proved in \cite{Ly} that ${\bf F_{\bullet}}(\LL;\underline{u})$ is a graded free resolution of $S/I$.

\begin{definition}
With the above notation, ${\bf F_{\bullet}}(\LL;\underline{u})$ is called the {\it Lyubeznik resolution}.
\end{definition}

Notice that the Lyubeznik resolution depends on the way in which we order the monomials $u_1,\ldots ,u_r$. Even if the Lyubeznik resolution tends to be much smaller than the Taylor resolution, it may happen (indeed it often happens) that it is not minimal for any order of the generators of $I$.  

\section{The algorithm}

Let $I\subseteq S=\kk[x_1,\ldots ,x_n]$ be a monomial ideal generated by monomials of positive degree $u_1,\ldots ,u_r$ ($r>0$). If $\sigma=\{i_1,\ldots ,i_s\}$ and $\tau=\{j_1,\ldots ,j_s\}$ are subsets of $[r]$ of the same cardinality, by $\sigma<\tau$ we mean that $\sigma$ is less than $\tau$ lexicographically, i.e. that there is $k$ such that $i_k<j_k$ and $i_p=j_p$ for all $p<k$. For any $i\in[r]$,
define the following set:
\begin{equation}\label{eq:lyu}
\Lyu_i:=\{\sigma\subseteq [r]:|\sigma|=i \mbox{ and }\exists \ k\in[r]: u_k|m_{\sigma_{>k}}\},
\end{equation}
where $\sigma_{>k}$ was defined in \eqref{eq:sigma>k}. Construct the graph $G(I)$ as follows:
\begin{itemize}
\item[(i)] The set of vertices of $G(I)$ is $\D_0\cup \D_1$ where 
\[\D_i:=\{\sigma\subseteq [r]:|\sigma|=i, \ \sigma\notin \Lyu_i\} \ \ \ \forall \ i=0,1.\]
\item[(ii)] There are no edges.
\end{itemize}
The purpose of the following algorithm is to iteratively transform the weighted graph $G(I)$ in a way that will provide 
the Betti table of $S/I$ (see sections \ref{RedRes} and \ref{SimplRes}). During the algorithm, the vertices of $G(I)$ 
will be always partitioned in 2 sets, namely $\D_i$ and $\D_{i+1}$. The main loop of the algorithm works as follows.
If either $\D_i$ or $\D_{i+1}$ is empty, the algorithm stops. Otherwise, the set of vertices $\D_{i}$ is stored and deleted, 
and a new set of vertices $\D_{i+2}$ is constructed together with a set of weighted edges $\D_{i+1}^{i+2}$, 
going from $\D_{i+2}$ to $\D_{i+1}$. Both sets $\D_{i+2}$ and $\D_{i+1}^{i+2}$ continually change, and eventually 
the edges of $\D_{i+1}^{i+2}$ will have weight $0$. At the end of each loop, the counter $i$ is increased by one.
In the worst case, when $i=\min\{n,r\}$, the set $\D_{i+1}$ is empty, so the algorithm stops.

\begin{remark}
With the above notation, the following inequalities follow by the Taylor's resolution:
\[\pdim(S/I)\leq \min\{r,n\} \ \ \ \mbox{and} \ \ \ \reg(S/I)\leq \pdim(S/I)\cdot  \max_i\{\deg(u_i)-1\}\]
In particular, the Betti table of $S/I$ has size at most $a(I)\times b(I)$, where $a(I):=\min\{r,n\}+1$ and $b(I):=\min\{r,n\}\cdot\max_i\{\deg(u_i)-1\}+1$.
\end{remark}


\begin{algorithm}
Let $\D_0, \D_1, a(I), b(I)$ be defined as above; set $i:=0$, $V_1:=\D_1$, and 
let $\BB$ be the $(a(I)\times b(I))$-matrix with all zero entries.

\begin{description}

\item[S1] Let $B_i$ be the $(a(I)\times b(I))$-matrix whose entries are all zero except for the $(i+1)$-th column: the $(j,i+1)$-th entry will be 
\[|\{\sigma\in \Delta_{i}:\deg(m_\sigma)=i+j-1\}|.\]
Analogously, let $B_{i+1}$ be the $(a(I)\times b(I))$-matrix whose entries are all zero except for the $(i+2)$-th column: the $(j,i+2)$-th entry will be 
\[|\{\sigma\in \Delta_{i+1}:\deg(m_\sigma)=i+j\}|.\]

\item[S2] {\bf If} $\bm{i=r-1}$ then set $\BB:=\BB+B_i+B_{i+1}$ and stop;\\
{\bf else if} $\bm{i=n}$ then set $\BB:=\BB+B_i$ and stop;\\
{\bf else if} $\bm{\D_i=\emptyset}$ then stop;\\
{\bf else if} $\bm{\D_{i+1}=\emptyset}$ then set $\BB:=\BB+B_i$ and stop;\\
{\bf else} set $\BB:=\BB+B_i$ and remove the set of vertices $\Delta_{i}$ from $G(I)$. 
Change the weighted graph $G(I)$ as described in the procedure {\bf The deformation of $G(I)$} below, 
set $i := i+1$ and go to step {\bf S1}.
\end{description}
\end{algorithm}

\bigskip

\noindent {\bf The deformation of $G(I)$}

\vspace{1mm}

\noindent Set $V_{i+2}:=\{\tau\cup\{v\}:\tau\in V_{i+1}, v\in [r]\setminus\tau\}$, $\D_{i+2}:=\emptyset$
and $\D_{i+1}^{i+2}:=\emptyset$. Iteratively transform $\D_{i+2}$ and $\D_{i+1}^{i+2}$ by lexicographically 
running on the elements $\sigma\in V_{i+2}$ in increasing order as follows:

\begin{itemize}
\item If $\sigma\in \Lyu_{i+2}$, then $V_{i+2}:=V_{i+2}\setminus \{\sigma\}$ and go to the next $\sigma$.

\item Otherwise, construct the following sets:
\begin{eqnarray*}
(\D_{i+1}^{\sigma})_1 &:=&\{\sigma\setminus \{v\} \in\D_{i+1} : v \in \sigma, m_{\sigma\setminus \{v\}} = m_\sigma
 \mbox{ and } (\sigma,\sigma \setminus \{v\})\notin \D_{i+1}^{i+2}\}\\
(\D_{i+1}^{\sigma})_2 &:=&\{\beta \in \Delta_{i+1}: (\sigma,\beta)\in\D_{i+1}^{i+2}\} \\
(\D_{i+1}^{\sigma})_3 &:=&\{\beta \in(\D_{i+1}^{\sigma})_2:[\sigma:\beta]\neq 0\}
\end{eqnarray*}
\item If $\D_{i+1}^{\sigma}:= (\D_{i+1}^{\sigma})_1\cup (\D_{i+1}^{\sigma})_3 = \emptyset$, then set 
$\D_{i+2}:=\D_{i+2}\cup\{\sigma\}$ and go to the next $\sigma$.
\item Otherwise, for every $\beta \in (\D_{i+1}^{\sigma})_1$, set
$[\sigma:\beta]:=\sign(\sigma \setminus \beta, \beta)$
(see \eqref{eq:sign}).\\
Let $\tau$ be the smallest element of $\D_{i+1}^{\sigma}$. Reset $\D_{i+1}^{\sigma}:=\D_{i+1}^{\sigma}\setminus \{\tau\}$. \\
Set $\D^{i+2}_{\tau}:=(\D^{i+2}_{\tau})_1\cup (\D^{i+2}_{\tau})_3$, where:
\begin{eqnarray*}
(\D^{i+2}_{\tau})_1 &:=&\{\alpha = \tau \cup \{v\}: v \in [r]\setminus \tau , \alpha > \sigma , \ u_v|m_\sigma , \ \alpha\notin \Lyu_{i+2}\} \\
(\D^{i+2}_{\tau})_2 &:=&\{\alpha \in \Delta_{i+2}:\alpha > \sigma , (\alpha,\tau)\in\D_{i+1}^{i+2}\} \\
(\D^{i+2}_{\tau})_3 &:=&\{\alpha \in(\D^{i+2}_{\tau})_2:[\alpha,\tau]\neq 0\} 
\end{eqnarray*}
If $\alpha \in (\D^{i+2}_{\tau})_1\setminus (\D^{i+2}_{\tau})_2$, define $[\alpha:\tau]:=\sign(\alpha \setminus \tau ,\alpha)$. \\
Set $\Delta_{i+2}:=\D_{i+2}\cup (\D^{i+2}_{\tau})_1$ and, for all $(\alpha , \beta)\in \D^{i+2}_{\tau}\times \D_{i+1}^{\sigma}$, procede as follows:
\begin{itemize}
\item If $(\alpha,\beta)\in\D_{i+1}^{i+2}$, then $[\alpha,\beta]:=[\alpha:\beta]-\frac{[\alpha:\tau]\cdot [\sigma:\beta]}{[\sigma:\tau]}$.
\item If $(\alpha,\beta)\notin\D_{i+1}^{i+2}$, then $\D_{i+1}^{i+2}:=\D_{i+1}^{i+2} \cup \{(\alpha,\beta)\}$ and:
\begin{itemize}
\item If $\alpha\supseteq \beta$, then $[\alpha,\beta]:=\sign(\alpha\setminus\beta,\alpha)-\frac{[\alpha:\tau]\cdot [\sigma:\beta]}{[\sigma:\tau]}$.
\item Otherwise, $[\alpha,\beta]:=-\frac{[\alpha:\tau]\cdot [\sigma:\beta]}{[\sigma:\tau]}$.
\end{itemize}
\end{itemize}
Set $\D_{i+2}:=\D_{i+2}\setminus (\{\sigma\}\cap\D_{i+2})$ and $\D_{i+1}:=\D_{i+1}\setminus \{\tau\}$. Put
\[\D_{i+1}^{i+2}:=\D_{i+1}^{i+2}\setminus (\{(\alpha,\tau):\alpha\in (\D^{i+2}_{\tau})_2\}\cup \{(\sigma,\beta):\beta\in (\D_{i+1}^{\sigma})_2\}) \]
and go to the next $\sigma$.

\end{itemize}

\bigskip

\begin{remark}
Note that the above algorithm can be modified to provide the multigraded Betti numbers of $S/I$: namely, given a vector $\bm{a}=(a_1,\ldots ,a_n)\in\ZZ^n$, 
\[\beta_{i,\bm{a}}(S/I)=|\{\sigma\in\D_i: m_{\sigma}=x_1^{a_1}\cdots x_n^{a_n}\}| \ \ \ \forall \ i\in \NN,\]
where $\D_i$ is taken at the step $i$.
\end{remark}

From the algorithm described above one can immediately infer some interesting upper/lower bounds for the depth/Castelnuovo-Mumford regularity of $S/I$. With the notation above, let $I=(u_1,\ldots ,u_r)\subseteq S=\kk[x_1,\ldots ,x_n]$. We say that a subset $\sigma\subseteq [r]$ is {\it critical (for $I$)} if: 
\[m_{\sigma}\neq m_{\sigma\setminus\{v\}} \ \forall \ v\in\sigma \ \ \ \mbox{ and } \ \ \ m_{\sigma}\neq m_{\sigma\cup\{v\}} \ \forall \ v\in[r]\setminus \sigma.\]
We introduce the following two invariants of a monomial ideal $I$:
\begin{eqnarray}
p(S/I):= & \max\{|\sigma|:\sigma \mbox{ is critical}\} \\
r(S/I):= & \max\{\deg(m_{\sigma})-|\sigma|:\sigma \mbox{ is critical}\}
\end{eqnarray}

\begin{proposition}
If $\sigma$ is a critical set for $I$, then $\beta_{|\sigma|,\deg(m_{\sigma})}(S/I)\neq 0$. In particular:
\[\depth(S/I)\leq n-p(S/I) \ \ \ \mbox{ and } \ \ \ \reg(S/I)\geq r(S/I).\]
\end{proposition}

As a special case, we recover a result of Katzman \cite{Ka}: let $G$ be a simple graph of $n$ vertices and $I=(x_ix_j:\{i,j\} \mbox{ is an edge of }G)$. He noticed that, if $G$ has an induced subgraph $H$ consisting of $t$ disjoint edges, then $\reg(S/I)\geq t$. Clearly, in such a situation the set $\sigma$ corresponding to the edges in $H$ is critical for $I$, so that $r(S/I)\geq \deg(m_{\sigma})-|\sigma|=t$.

\begin{remark}\label{rem:simplicial homology}
To compute the multigraded Betti numbers of a monomial ideal, we can always reduce ourselves to square-free monomial ideals by polarization. A square-free monomial ideal $I=(u_1,\ldots ,u_r)\subseteq S$ is associated to a simplicial complex $\D$ on $n$ vertices via the Stanley-Reisner correspondence. Thanks to the Hochster's formula (cf. \cite[Corollary 5.12]{MS}), to know the multigraded Betti numbers of $S/I$ is equivalent to know the simplicial homology with coefficients in $\kk$ of all the $2^n$ induced subcomplexes of $\D$. 

If one is only interested in the simplicial homology of $\D$, it is enough to know the multigraded Betti numbers $\beta_{i,(1,\ldots ,1)}(S/I)$, more precisely:
\[\dim_{\kk}\widetilde{H}_i(\D;\kk)=\beta_{n-i-1,(1,\ldots ,1)}(S/I)\]
Being our algorithm ``local'', to know such multigraded Betti numbers is not necessary to compute the minimal free resolution of $S/I$ completely. In fact, it is possible to implement a faster version of the algorithm aimed at the computation of the simplicial homology with coefficients in a field. Already several algorithms are available for this task, with applications, for example, in network theory. Nevertheless, we wish to study the potentiality of such an algorithm in a forthcoming paper.

To give a preview of how such an algorithm would work, let $\{G_1,\ldots ,G_r\}$ be the set of minimal nonfaces of $\D$. Essentially, one should perform the ``deformation of $G(I)$" only on the sets of vertices
\[V_i=\{\sigma\subseteq [r]\setminus \Lyu_i:|\sigma|=i  \mbox{ and } \bigcup_{j\in\sigma}G_j=[n]\} \ \ \ i=n-\dim \D ,\ldots ,n.\]
Then $\dim_{\kk}\widetilde{H}_i(\D;\kk)$ will be the cardinality of the set $\D_{n-i-1}$ (with the same notation of the algorithm) at the end of the deformation.
A remarkable feature is that to compute the simplicial homology in this way, one needs as input the minimal nonfaces of the simplicial complex, rather than its facets. This fact is convenient in certain situations. For example, in many cases is of interest to compute the homology of the clique complex of a given graph: while the facets of such a simplicial complex are the maximal cliques of the graph, its minimal nonfaces are simply the pair of nonadjacent vertices, which are easier to store if only the graph is given (as it typically happens in network theory).
\end{remark}


\section{Experiments}

In this section we present some examples to show the effectiveness of our algorithm. 
We implemented it using the C++ language and some routines available in CoCoALib,
a GPL C++ library which is the mathematical kernel for the computer algebra system CoCoA-5
\cite{cocoa}. All computations are performed on an Intel Core 2 i5 processor (at 1.4 GHz)
using both our new algorithm and the function \emph{BettiDiagram} included in CoCoA-4.
We are not comparing our timings with CoCoA-5 itself because the implementation for minimal free resolutions 
has not yet been optimized in the current version CoCoA-5.1.1.
We perform three experimental tests: in the first case we consider monomial ideals of degree $d=2$
(results summarized in Table 1.(a)); in the second case we consider monomial ideals of 
degree $d=4$ (results summarized in Table 1.(b)); in the third case we consider monomial 
ideals of degree $d$, with $5 \le d \le 8$ (results summarized in Table 1.(c)). 

Tables 1.(a), 1.(b) and 1.(c) consist of four columns: the first column, labeled with 
$n$, contains the number of variables of the polynomial ring; the second column, labeled with~$r$, 
contains the number of minimal generators of the monomial ideal; the last two columns, labeled with ``New"
and ``CoCoA", contain the average time  to compute the Betti tables of $10$
monomial ideals using our new algorithm and the software CoCoA-4 respectively. 
In all tested numerical experiments our algorithm reveals to be more efficient than CoCoA-4, 
since it returns the Betti tables in a shorter computational time. 
We also performed some comparisons of our results with computations done
using the software Macaulay-2 (see \cite{M2}), in which case the computational    
 timings are comparable, with variations depending on the size of the example. In our opinion, though, by accurately refining the implementation, the efficiency of our algorithm has substantial room for improvement.
 
 \begin{table}[htb]
\centering
\subtable[Monomial ideals of degree $d=2$. \label{table1}]{%
\begin{tabular}{|c|c|c|c|c|c|}\cline{3-4}
\multicolumn{2}{c|}{}& \multicolumn{2}{|c|}{\bf Timings}\\
 \hline
 $n$ & $r$ & New &  CoCoA\\
 \hline
 $10$ & $8$ & 0.001 &  0.02 \\
 $10$ & $10$ & 0.005 &  0.04 \\
 $10$ & $12$ & 0.014 &  0.11\\
 \hline
 $15$ & $12$ & 0.023 &  0.19 \\
 $15$ & $15$ & 0.177 & 1.32 \\
 $15$ & $18$ & 0.512  & 2.33 \\
 \hline
 $18$ & $18$ & 2.111  & 13.14 \\
 $18$ & $20$ &  9.608   & 32.87 \\
\hline
\end{tabular}}
\qquad 
\subtable[Monomial ideals of degree $d=4$. \label{table2}]{%
\begin{tabular}{|c|c|c|c|c|c|}\cline{3-4}
\multicolumn{2}{c|}{}& \multicolumn{2}{|c|}{\bf Timings}\\
 \hline
 $n$ & $r$ & New & CoCoA\\
 \hline
 $10$ & $8$ & 0.002 &  0.02 \\
 $10$ & $10$ & 0.004 &  0.05 \\
 $10$ & $12$ & 0.012 & 0.11 \\
 \hline
 $15$ & $12$ & 0.029  & 0.43 \\
 $15$ & $15$ & 0.208  & 2.91 \\
 $15$ & $18$ &  0.786 & 9.37 \\
 \hline
 $18$ & $18$ &  3.282  & 78.88\\
 $18$ & $20$ &        24.194      & 158.27\\
\hline
\end{tabular}
}
\qquad 
\subtable[Monomial ideals of degree $d$, with $5 \le d \le 8$. \label{table2}]{%
\begin{tabular}{|c|c|c|c|c|c|}\cline{3-4}
\multicolumn{2}{c|}{}& \multicolumn{2}{|c|}{\bf Timings}\\
 \hline
 $n$ & $r$ & New & CoCoA\\
 \hline
 $10$ & $8$ & 0.002 & 0.04  \\
 $10$ & $10$ & 0.007 & 0.09 \\
 $10$ & $12$ & 0.014  & 0.23 \\
 \hline
 $15$ & $12$ & 0.035  & 0.81 \\
 $15$ & $15$ & 0.261 & 7.51 \\
 $15$ & $18$ & 1.247 & 58.84\\
 \hline
 $18$ & $18$ &  4.064 &  271.5 \\
 $18$ & $20$ &  63.00 &   482.73\tablefootnote{this average is made on 7 tests out of 10, since in 3 experiments the computation ran out of memory}  \\
\hline
\end{tabular}}
\caption{\small{Tests on monomial ideals of degree $d=2$, $d=4$ and $5 \le d \le 8$}}\label{table}
\end{table}

\section{A vanishing inspired by the algorithm}

In the last section of the paper we prove a vanishing result arisen while developing the algorithm. 
Before stating it, for the convenience of the reader we recall how to get the {\it Mayer-Victories spectral sequence} for simplicial homology:

\begin{remark}
Let $\D$ be a simplicial complex on $[n]$, and $\D_0,\ldots ,\D_s$ be subcomplexes such that
\[\D=\bigcup_{i=0}^s\D_i.\]
For any sequence of integers $0\leq a_0<\ldots <a_p\leq s$, let us denote by 
\[\D_{a_0,\ldots,a_p}=\bigcap_{k=0}^p\D_{a_k}.\]
Fixed a ring $R$, for any simplicial complex $\Gamma$ we denote by ${\bf C_{\bullet}}(\Gamma;R)$ the complex of $R$-modules defined as in Section \ref{SimplRes} in the particular case that $R$ was the field $\kk$,
and by ${\bf C^{\bullet}}(\Gamma;R):=\Hom_R({\bf C_{\bullet}}(\Gamma;R))$. One can check that the following is an exact sequence of complexes of $R$-modules:
\[0\rightarrow {\bf C^{\bullet}}(\D;R)\xrightarrow{d^0} \C^{\bullet,0}(R)\xrightarrow{d^1} \C^{\bullet,1}(R)\xrightarrow{d^2} \cdots \xrightarrow{d^{s-1}} \C^{\bullet,s-1}(R)\xrightarrow{d^s} \C^{\bullet,s}(R)\rightarrow 0,\]
where $\C^{\bullet,p}(R)=\bigoplus_{a_0<\ldots <a_p}{\bf C^{\bullet}}(\D_{a_0,\ldots ,a_p};R)$ and, for a given $j$, if $\alpha=(\alpha_{a_0,\ldots,a_p})_{a_0<\ldots<a_p}$ is an element of $\C^{j,p}(R)$, the $(b_0,\ldots ,b_{p+1})$th component of $d^{p+1}(\alpha)\in \C^{j,p+1}(R)$ is:
\[\sum_{k=0}^{p+1}(-1)^{k}\left.\left(\alpha_{b_0,\ldots ,\widehat{b_k},\ldots ,b_{p+1}}\right)\right|_{C^j(\D_{b_0,\ldots ,b_{p+1}})}\]
(considering $\D$ as $\D_{\emptyset}$).
Let us consider the double complex of $R$-modules $\C^{\bullet,\bullet}(R)$ where the differentials 
\[\partial^p:\C^{p-1,q}(R)\rightarrow \C^{p,q}(R),\]
for a given natural number $q$, are the ones used to define the simplicial cohomology of $\D_{a_0,\ldots,a_q}$ with coefficients in $R$. By the general theory of spectral sequences arising from double complexes (c.f. \cite[Chapter III.7, Proposition 10]{GM}), we get the following spectral sequence:
\begin{eqnarray}\label{MVspectral}
\bigoplus_{0\leq a_0<\ldots <a_p\leq s}H^q\left(\D_{a_0\ldots ,a_p};R\right)\implies H^{p+q}(\D;R).
\end{eqnarray}
\end{remark}

Given two elements $\bm{a}=(a_1,\ldots ,a_n)$ and $\bm{b}=(b_1,\ldots ,b_n)$ of $\ZZ^n$, for $\bm{a}\leq \bm{b}$ we mean $a_i\leq b_i$ for all $i\in[n]$. Furthermore, for $\bm{a}< \bm{b}$ we mean $\bm{a}\leq \bm{b}$ and $\bm{a}\neq \bm{b}$.

\begin{theorem}\label{thm:subadditivity}
Let $I=(u_1,\ldots ,u_r)$ be a monomial ideal of $S=\kk[x_1,\ldots ,x_n]$, and let $\bm{a_{\ell}}=\deg_{\ZZ^n}(u_{\ell})$ for all $\ell\in[r]$. For any $\ell\in[r]$, $i\in \NN$ and $\bm{a}\in\ZZ^n$:
\[\beta_{i,\bm{b}}(S/I)=0 \ \ \forall \ \bm{a}\leq \bm{b}<\bm{a}+\bm{a_{\ell}} \ \ \ \implies \ \ \ \beta_{i+1,\bm{a}+\bm{a_{\ell}}}(S/I)=0. \] 
In particular, if $d=\max_{\ell}\{\deg_{\ZZ}(u_{\ell})\}$, for any $j\in\ZZ$:
\[\beta_{i,k}(S/I)=0 \ \ \forall \ k=j,\ldots ,j+d-1 \ \ \ \implies \ \ \ \beta_{i+1,j+d}(S/I)=0. \] 
\end{theorem}
\begin{proof}
It is not difficult to see that, by polarizing, one can reduce himself to the square-free case (though one has to be a bit careful because by polarizing the multi-degree changes). In such a situation, for any $s\in\NN$, $\beta_{s,\bm{c}}(S/I)\neq 0$ only if $\bm{c}=(c_1,\ldots ,c_n)$ is a square-free vector, that is $c_j\in \{0,1\}$ for any $j\in [n]$. It is convenient to think at such square-free vectors as subsets of $[n]$, namely 
\[\bm{c} \longleftrightarrow \sigma(\bm{c})=\{j\in[n]:c_j=1\}.\]
Notice that $\bm{c_1}\leq \bm{c_2}$ if and only if $\sigma(\bm{c_1})\subseteq \sigma(\bm{c_2})$.
Let $\alpha$, $\alpha_{\ell}$ and $\gamma=\alpha\cup\alpha_{\ell}$ be the subsets of $[n]$ corresponding to, respectively, $\bm{a}$, $\bm{a_{\ell}}$ and $(\bm{a}+\bm{a_{\ell}}) (\mathrm{mod}2)$. Furthermore, let $\Delta$ be the simplicial complex such that $I=I_{\Delta}$ and set $\Gamma=\Delta_{\gamma}$. Notice that our vanishing hypothesis can be written as:
\[\beta_{i,\beta}(S/I)=0 \ \ \forall \ \alpha\subseteq \beta \subsetneq \gamma.\] 
By Hochster's formula, this is equivalent to say:
\begin{equation}\label{eq:vanishing}
H^{|\beta|-i-1}(\Gamma_{\beta};\kk)=0 \ \ \forall \ \alpha\subseteq \beta \subsetneq \gamma.
\end{equation}
Since $\alpha_{\ell}$ is not a face of $\Delta$,
\[\bigcup_{v\in \alpha_{\ell}}\Gamma_{\gamma\setminus\{v\}}=\Gamma.\]
Obviously, for all subsets $\delta$ of $\alpha_{\ell}$:
\[\bigcap_{v\in \delta}\Gamma_{\gamma\setminus\{v\}}=\Gamma_{\gamma\setminus\delta}.\]
Therefore the Mayer-Vietoris spectral sequence \eqref{MVspectral} with respect to the covering $(\Gamma_{\gamma\setminus\{v\}})_{v\in\alpha_{\ell}}$ of $\Gamma$ is:
\[\bigoplus_{\substack{\delta\subseteq \alpha_{\ell} \\ |\delta|=p+1}}H^q(\Gamma_{\gamma\setminus\delta};\kk)\implies H^{p+q}(\Gamma;\kk).\]
So, by \eqref{eq:vanishing} we infer that $H^{|\gamma|-i-2}(\Gamma;\kk)=0$. We conclude by using Hochster's formula that
\[\beta_{i+1,\gamma}(S/I)=0.\]
\end{proof}

\begin{remark}
As a particular case of the above theorem, we get:
\[
\beta_{i,k}(S/I)=0 \ \ \forall \ k\geq j \ \ \ \implies \ \ \ \beta_{i+1,k+d}(S/I)=0 \ \ \forall \ k\geq j. \\
\]
This fact has recently been proved by Herzog and Srinivasan in \cite[Corollary 4]{HS}. A further consequence is the following result of Fern\'andez-Ramos and Gimenez \cite[Theorem 2.1]{FG}: if $I$ is the edge ideal of a graph, then $\beta_{i,j}(S/I)=\beta_{i,j+1}(S/I)=0  \implies  \beta_{i+1,j+2}(S/I)=0$.
\end{remark}


\begin{thebibliography}{BMS2}

\bibitem[BPS98]{BPS}
D.~Bayer, I.~Peeva, B.~Sturmfels. \emph{Monomial resolutions}. Mathematical Research Letters, Vol. 5 (1998), pp. 31-46.

\bibitem[BS98]{BS}
D.~Bayer, B.~Sturmfels, \emph{Cellular resolutions of monomial modules}, 
J. Reine Angew. Math. 502 (1998), 123-140. 


\bibitem[CoCoA]{cocoa} CoCoATeam, \emph{CoCoA: a system for doing Computations in Commutative Algebra}, available at \url{http://cocoa.dima.unige.it}.

\bibitem[Singular]{singular}
W.~Decker, G.-M.~Greuel, G.~Pfister, H.~Sch\"onemann, \emph{Singular: A computer algebra system for polynomials computations}, available at \url{http://www.singular.uni-kl.de}.

\bibitem[DS98]{DS}  P.~Diaconis, B.~Sturmfels, \emph{Algebraic algorithms for sampling from conditional distributions}, 
Ann. Statist. 26 (1998), no. 1, 363-397.

\bibitem[D11]{D} I. H.~Dinwoodie, \emph{Syzygies for Metropolis base chains}, Linear Algebra and its Applications, 434 (2011)
2176-2186.

\bibitem[Ei95]{Ei}  D.~Eisenbud, \emph{Commutative algebra. With a view toward algebraic geometry}, Graduate Texts in Mathematics, 150, Springer-Verlag, New York (1995).

\bibitem[FG14]{FG}  O.~Fern\'andez-Ramos, P.~Gimenez, \emph{Regularity 3 in edge ideals associated to bipartite graphs}, J. Algebraic Combin. 39, no. 4, 919-937 (2014).

\bibitem[GM03]{GM}
S.~Gelfand, YU.~Manin, \emph{Methods of Homological Algebra}, Springer Monographs in Mathematics (second edition), Springer-Verlag Berlin Heidelberg (2003).


\bibitem[Macaulay2]{M2}
D.~Grayson, M.~Stillman, \emph{Macaulay2, a software system for research in algebraic geometry}, available at \url{http://www.math.uiuc.edu/Macaulay2/}.

\bibitem[JW09]{JW}
M.~J\"ollenbeck, V. Welker, \emph{Minimal Resolutions Via Algebraic Discrete Morse Theory}, Memoirs of the American Mathematical Society, Vol. 197, Nr. 923 (2009).

\bibitem[HS14]{HS}
J.~Herzog, H.~Srinivasan, \emph{A note on the subadditivity problem for maximal shifts in free resolutions}, to appear in MSRI Proceedings.

\bibitem[Ka06]{Ka} M. Katzman, \emph{Characteristic-independence of Betti numbers of graph ideals}, J. Combin. Theory Ser. A 113, n. 3, pp. 435-454 (2006).

\bibitem[Ly88]{Ly}
G.~Lyubeznik, \emph{A new explicit free resolution of ideals generated by monomials in an $R$-sequence}, Journal of Pure and Applied Algebra, Vol. 51 (1988), 193-195.

\bibitem[MS04]{MS} E.~Miller, B.~Sturmfels,  \emph{Combinatorial commutative algebra}, Graduate Texts in Mathematics, Springer, Vol. 227 (2004). 

\bibitem[SW09]{SW} E.~S\'aenz-de-Cabez\'on, H.~Wynn, \emph{Betti numbers and minimal free resolutions for multi-state system reliability bounds}, Journal of Symbolic Computation, Vol. 44, no. 9 (2009), 1311-1325.  
 
\bibitem[Sc80]{Sc1} F.-O.~Schreyer, \emph{Die Berechnung von Syzygien mit dem verallgemeinerten Weierstrassschen Divisionssatz}. Diplomarbeit, Hamburg (1980).

\bibitem[Sc91]{Sc2} F.-O.~Schreyer, \emph{A standard basis approach to syzygies of canonical curves}, J. reine angew. Math. 421, 83-123 (1991).

\bibitem[Ta60]{Ta}
D.~Taylor, \emph{Ideals generated by monomials in an $R$-sequence}, PhD thesis, University of Chicago (1960).

\bibitem[Ya15]{Ya} A.~A.~Yazdan Pour, \emph{Candidates for non-zero Betti numbers of monomial ideals}, available at \url{http://arxiv.org/abs/1507.07188} (2015).

\end{thebibliography}
\end{document}